\documentclass{amsart}

\usepackage{amssymb}
\usepackage{amsthm}
\usepackage{enumerate}

\newtheorem{thr}{Theorem}
\newtheorem{lem}{Lemma}

\newtheorem{cor}{Corollary}
\newtheorem{prop}{Proposition}
\newtheorem{obs}{Observation}

\theoremstyle{remark}
\newtheorem*{remark}{Remark}

\theoremstyle{definition}
\newtheorem*{acknowledgment}{Acknowledgments}

\begin{document}
\title{Bernstein and Kantorovich polynomials diminish the $\Lambda$-variation}

\author[Klaudiusz Czudek]{Klaudiusz Czudek}
\address{Institute of Mathematics, Faculty of Mathematics, Physics, and Informatics, University of Gda\'{n}sk, 80-308 Gda\'{n}sk, Poland}
\email{klaudiusz.czudek@gmail.com}

\begin{abstract}
We prove the $\Lambda$-variation diminishing property of the Bernstein and Kantorovich polynomials. Next we apply this result to characterize the space $C\Lambda BV_c$ as the closure of the space of polynomials in the $\|\cdot\|_\Lambda$ norm. A new proof of the separability of $C\Lambda BV_c$ is given.
\end{abstract}

\date{\today}
\maketitle
\section{Introduction}
It is a well-known fact that every continuous function defined on the interval $[0,1]$ can be approximated by polynomials. Several proofs of this theorem were published, the first one by Karl Weierstrass and later by Henri Lebesgue, Marshall H. Stone and Sergei Bernstein, among others. This one given by Bernstein is particulary important in our paper. It establishes that for a continuous function $f:[0,1]\rightarrow\mathbb{R}$ polynomials of the form:
$$B_nf(x)=\sum_{k=0}^n f\left(\frac{k}{n}\right){n \choose k}x^k(1-x)^{n-k}$$
tend to $f$ uniformly on $[0,1]$. Properties of these polynomials, called the Bernstein polynomials of the function $f$, were intensively studied in the last century, also in the case when the initial function $f$ is not continuous. Basic facts related to this topic may be found in the classic positions \cite{Lorentz_Constructive_approximation} or \cite{Lorentz_Bernstein_polynomials}. The Bernstein polynomials have also appeared attractive in numerous applications. We would like to suggest the recent paper \cite{Farouki_Bernstein_polynomial_basis} as an overview of these achievements. Observe that one may look at the Bernstein polynomials as a family of linear operators $f\mapsto B_nf$ what is the reason why we use terms "Bernstein operators" and "Bernstein polynomials" interchangeably.

In the twentieth century a lot of generalizations of the Bernstein operators appeared, so-called Bernstein-type operators. One of the most significant are the Sz\'asz-Mirakyan, the Baskakov and the Kantorovich operators. The last one has an important place in our paper. For a function $f$, integrable on the interval $[0,1]$, one can define the $n$-th Kantorovich polynomial:
$$K_nf(x)=\sum_{k=0}^n {n \choose k}x^k(1-x)^{n-k}(n+1)\int_{\frac{k}{n+1}}^{\frac{k+1}{n+1}}f(t)dt.$$
In comparsion with the Bernstein polynomials, the value of $f$ at the point $\frac{k}{n}$ is replaced with the mean value of $f$ on the interval $[\frac{k}{n+1},\frac{k+1}{n+1}]$. In 1930 Kantorovich proved that $\|K_nf-f\|_1\to 0$ for any integrable function $f$, where $\|\cdot\|_1$ denotes the standard norm in the space of integrable functions $L^1$. In \cite{Lorentz_Constructive_approximation} we may find other approximation theorems related to Kantorovich polynomials, including $L^p$ and supremum norms.

Among a number of properties of Bernstein polynomials we would like to distinguish the variation diminishing property, proved by G.G. Lorentz in \cite{Lorentz_Zur_theorie_polynome_Bernstein}. It states that for an arbitrary function $f:[0,1]\rightarrow\mathbb{R}$ of bounded variation the variation of the Bernstein polynomial of the function $f$ is not greater than the variation of $f$, i.e. $V(B_nf)\leq V(f)$. Recall that the variation of a function $f:[0,1]\rightarrow\mathbb{R}$ is the quantity:
$$V(f)=\sup \sum_{i=1}^{n-1} \left|f(x_{i+1})-f(x_{i})\right|$$
where supremum is taken over all finite sequences $0\leq x_1<...<x_n\leq 1$. This property was also proven in the case of other Bernstein-type operators, in particular in the case of Kantorovich polynomials. Additionaly Lorentz showed that $\|B_nf-f\|_{BV}\to 0$ if and only if $f$ is absolutely continuous, where $\|\cdot\|_{BV}$ denotes the standard norm in the space $BV$ of functions of bounded variation $\|u\|_{BV}=V(u)+|u(0)|$. Proofs of all these facts and rates of approximation may be found in \cite{Vinti_Congergence_in_variation_rates_of_approximation_Bernstein_polynomials}.  Let us mention here that the variation diminishing property is called in \cite{Vinti_Congergence_in_variation_rates_of_approximation_Bernstein_polynomials} the variation detracting property, while the term "variation diminishing" is used to express the statement that a number of zeros of a Bernstein polynomial of a function $f$ with multiplicities counted is not greater than the number of sign changes of the function $f$. Our terminology follows from \cite{Adell_Cal_Bernstein_operators_diminish_variation}.

The variation diminishing property of the Bernstein polynomials was also proven for certain generalization of the regular variation, so-called $\varphi$-variation. Given a convex, nondecreasing function on $[0,\infty)$ continuous in 0 and satisfying $\varphi(0)=0$ define the $\varphi$-variation of $f:[0,1]\rightarrow\mathbb{R}$ as the number:
$$V_{\varphi}(f)=\sup \sum_{i=1}^{n-1} \varphi\left(|f(x_{i+1})-f(x_{i})|\right)$$
where supremum is taken over all finite sequences $0\leq x_1<...<x_n\leq 1$. We say that a function $f$ is of bounded $\varphi$-variation if $V_\varphi(f)<\infty$. The space of all functions of bounded $\varphi$-variation is denoted by $\varphi BV$. Basic facts about $\varphi$-variation may be found in \cite{Musielak_Generalized_variations}. J. A. Adell and J. de la Cal proved in \cite{Adell_Cal_Bernstein_operators_diminish_variation}, using probabilistic representation of the Bernstein operators, the $\varphi$-variation diminishing property, i.e. that $V_\varphi(B_nf)\leq V_\varphi(f)$ for $f\in\varphi BV$. Let us mention that Lorentz's technique from \cite{Lorentz_Zur_theorie_polynome_Bernstein} could not be applied here, since he used the fact that every function of bounded variation is a difference of two monotone functions, what does not remain valid in the case of the $\varphi$-variation.

There exists also other generalization of the regular variation, the $\Lambda$-variation, introduced by Daniel Waterman in \cite{Waterman_On_convergence_Fourier_series_generalized_variation} for the sake of its applicability to the study of the uniform convergence of Fourier series. We say that a sequence $\Lambda=(\lambda_n)_n$ of positive reals is a $\Lambda$-sequence if $\Lambda$ is nondecreasing and $\sum_n\frac{1}{\lambda_n}=\infty$. Given an interval $I=[a,b]$ and a function $f$ denote $|f(I)|=|f(b)-f(a)|$. A finite or infinite sequence $(I_n)_n$ of closed intervals is called nonoverlapping if for every $k\not=n$ intervals $I_n, I_k$ intersect at most at the endpoints. We define the $\Lambda$-variation of a function $f:[0,1]\rightarrow\mathbb{R}$ as the number:
$$V_\Lambda(f)=\sup\sum_{n}\frac{|f(I_n)|}{\lambda_n},$$
where supremum is taken over all sequences (finite or infinite) of nonoverlapping intervals $(I_n)_n$ contained in $[0,1]$. Naturally, we say that function $f$ is of bounded $\Lambda$-variation if $V_\Lambda(f)<\infty$. Notice that for $\Lambda=(1,1,1,....)$ we get that $V_\Lambda(f)=V(f)$.

Given a finite sequence of nonoverlapping closed intervals $\mathcal{I}=\langle I_1,...,I_m\rangle$ contained in $[0,1]$, $\Lambda$-sequence $\Lambda$ and a function $f:[0,1]\rightarrow\mathbb{R}$ define:
$$\sigma_\Lambda\left(f,\mathcal{I}\right)=\sum_{k=1}^m\frac{|f(I_k)|}{\lambda_k}$$
It is readily seen that we may define the $\Lambda$-variation of the function $f:[0,1]\rightarrow\mathbb{R}$ equivalently as the number:
$$V_\Lambda(f)=\sup\sigma_\Lambda\left(f,\mathcal{I}\right)$$
where supremum is taken over all finite sequences $\mathcal{I}$ of nonoverlapping closed intervals contained in $[0,1]$. Note that the symbol $\sigma_\Lambda$ was used in the papers of Prus-Wi{\'s}niowski but with a slightly different meaning (cf. \cite{Prus_Variation_Hausdorff_distance}, \cite{Prus_Separability_continuous_functions_contunuous_Lambda_variation}, \cite{Prus_Absolute_continuity}). 

For a fixed $\Lambda$-sequence $\Lambda$, the space of functions of bounded $\Lambda$-variation, denoted by $\Lambda BV$, is a linear space, which may be equipped with the norm $\|f\|_\Lambda=V_\Lambda(f)+|f(0)|$. Note that the $\Lambda$-variation is only a seminorm. Moreover, the normed space $(\Lambda BV, \|\cdot\|_\Lambda)$ is a Banach space (Section 3. in \cite{Waterman_On_Lambda_bounded_variation}). Obviously, every function from $\Lambda BV$ is bounded. Using simple arguments one can show that convergence in the $\|\cdot\|_\Lambda$ norm implies convergence in the supremum norm. Notice that we may deduce from this fact that the space $C\Lambda BV$ of continuous functions of bounded $\Lambda$-variation is a closed subspace of $\Lambda BV$.

Further, if $\Lambda$ is a proper $\Lambda$-sequence, i.e. $\lambda_n\to\infty$, then a function $f$ is said to be continuous in $\Lambda$-variation if $\lim_{m\to\infty} V_{\Lambda_{(m)}}(f)\to 0$, where $\Lambda_{(m)}$ denotes a $\Lambda$-sequence obtained from $\Lambda$ by omission of the first $m$ terms: $\Lambda_{(m)}=(\lambda_{m+1}, \lambda_{m+2},...)$. We denote by $\Lambda BV_c$ and $C\Lambda BV_c$ the space of functions continuous in $\Lambda$-variation and the space of continuous functions that are continuous in $\Lambda$-variation, respectively. The concept of continuity in $\Lambda$-variation was introduced by Waterman in \cite{Waterman_Summability_Fourier_series_bounded_variation} in connection with $(C,\beta)$-summability of Fourier series. It gained more importance later, since it turned out that Fourier series of these functions have much more interesting properties (cf. \cite{Avdispahic_Generalized_variation_Fourier_series}).

The main purpose of this paper is to generalize Lorentz's results from \cite{Lorentz_Zur_theorie_polynome_Bernstein} to the $\Lambda$-variation. In the first section we prove that the Bernstein polynomials diminish the $\Lambda$-variation. In the proof we use idea from \cite{Adell_Cal_Bernstein_operators_diminish_variation} to use the probabilistic representation of the Bernstein operators. We conclude from this fact that the Kantorovich operators also diminish the $\Lambda$-variation (as far as we know, there does not exist a probabilistic representation of the Kantorovich operators satisfying the desired condition). In the third section we deal with related notions of $\Lambda$-variation. Finally, in the fourth section we apply the $\Lambda$-variation diminishing property to characterize $C\Lambda BV_c$ as the closure of the space of polynomials in the $\|\cdot\|_\Lambda$ norm and therefore we partially answer the question asked by Waterman in \cite{Waterman_On_Lambda_bounded_variation} about characterization of the space $\Lambda BV_c$. Among others, this question has been already answered by Prus-Wi\'sniowski but we will look at the problem from different point of view. A new proof of the separability of $C\Lambda BV_c$ and a new characterization of compactness in this space are also given.

\section{The $\Lambda$-variation diminishing property}
Let $(\Omega, \mathfrak{M}, P)$ be a probability space, $(X_k)_k$ a sequence of independent random variables defined on $\Omega$ and uniformly distributed on the interval $[0,1]$. For every natural $k$ and $x\in [0,1]$ the indicator function $I_{(X_k\leq x)}$ is a random variable on $\Omega$ and therefore:
$$S_n(x)=\sum_{k=0}^n I_{(X_k\leq x)}$$
is a random variable. If $x\in [0,1]$ and $k$ is a natural number, then the indicator function $I_{(X_k\leq x)}$ takes two values: $1$ with probability $x$ and $0$ with probability $1-x$. It is quite clear, that $S_n(x)$ has the binomial distribution with parameters $n, x$ i.e. takes value $k$, $0\leq k\leq n$, with probability ${n\choose k}x^k(1-x)^{n-k}$. Finally, define the random variable:
$$Z_n^x=\frac{S_n(x)}{n}$$
Notice that $Z_n^x$ takes values $\frac{k}{n}$, for $0\leq k\leq n$ with probability ${n\choose k}x^k(1-x)^{n-k}$.

Take any Borel measurable function $f:[0,1]\rightarrow\mathbb{R}$. The random variable $f(Z^x_n)$ is a simple, measurable function which takes values $f(\frac{k}{n})$, for $0\leq k\leq n$ with probability ${n\choose k}x^k(1-x)^{n-k}$  and therefore:
\begin{equation}
\label{probabilistic_representation}
\mathbb{E}f(Z^x_n)=\sum_{k=0}^n f\left(\frac{k}{n}\right){n\choose k}x^k(1-x)^{n-k}=B_nf(x)
\end{equation}
where $B_nf$ denotes the $n$-th Bernstein polynomial of the function $f$. Inequality $x_1\leq x_2$ implies that the set $(X_k\leq x_1)$ is contained in the set $(X_k\leq x_2)$ for every natural $k$, hence $I_{(X_k\leq x_1)}\leq I_{(X_k\leq x_2)}$ and eventually:
\begin{equation}
\label{inequality_probabilistic_representation}
Z_n^{x_1}\leq Z_n^{x_2}
\end{equation}
for $x_1\leq x_2$, what is a crucial property in our paper. The probabilistic representation was, according to \cite{Adell_Cal_Bernstein_operators_diminish_variation}, firstly used by Lindvall in \cite{Lindvall_Bernstein_polynomials_law_of_large_numbers} to deduce some properties of the Bernstein polynomials. As an example let us take a look at the following proposition which will be used in the next section.

\begin{prop}
\label{preservation_monotonicity}
If $f$ is nondecreasing (nonincreasing), then $B_nf$ is nondecreasing (nonincreasing).
\end{prop} 
\begin{proof} Trivial from (\ref{probabilistic_representation}) and (\ref{inequality_probabilistic_representation}).
\end{proof}
\noindent More theorems and proofs in this spirit, as the preservation of convexity, may be found in \cite{Adell_Cal_Using_stochastic_processes_for_studying_Bernstein_operators} or \cite{Lindvall_Bernstein_polynomials_law_of_large_numbers}.

We are ready to show the $\Lambda$-variation diminishing property of the Bernstein polynomials. The proof of the theorem below follows the lines of the proof of the analogous result in \cite{Adell_Cal_Bernstein_operators_diminish_variation}. In fact, the authors of the paper \cite{Adell_Cal_Bernstein_operators_diminish_variation} proved this statement for a wide class of Bernstein-type operators with a probabilistic representations, i.e. represented as the mean value of $f$ composed with some double-indexed stochastic process $\{Z_n^x: n\in\mathbb{N}, x\in[0,1]\}$ satisfying (\ref{inequality_probabilistic_representation}). Moreover, they do not restrict only to operators acting on spaces of functions defined on the interval $[0,1]$. Examples of such operators are listed in \cite{Adell_Cal_Bernstein_operators_diminish_variation}.

Note that:
$$V_\Lambda(f)=\sup\sum_{i=1}^{m-1}\frac{|f(x_{j+1})-f(x_j)|}{\lambda_{\beta(j)}}$$
where supremum is taken over all finite partitions $0\leq x_1\leq x_2\leq ...\leq x_m\leq 1$ of the interval $[0,1]$ and permutations $\beta$ of the set $\{1,2,...,m-1\}$. This observation is used in several papers to simplify a notation of proofs (for example \cite{Prus_Variation_Hausdorff_distance}, \cite{Prus_Separability_continuous_functions_contunuous_Lambda_variation}, \cite{Prus_Absolute_continuity}, the proof of this fact in \cite{Bugajewski_Nonlinear_operators}). 

\begin{thr}
\label{Bernstein_operators_diminish_lambda_variation}
If $\Lambda$ is a $\Lambda$-sequence, $f\in\Lambda BV$, then $V_\Lambda(B_nf)\leq V_\Lambda(f)$.
\end{thr}
\begin{proof}
Function $f$ has one-sided limits in every point of its domain (see Theorem 4 in \cite{Perlman_Functions_of_generalized_variation}), what implies that it is a Borel function and hence representation (\ref{probabilistic_representation}) holds for $B_nf$. Take $0\leq x_1\leq x_2\leq ...\leq x_m\leq 1$. From (\ref{inequality_probabilistic_representation}) we have that $Z_n^{x_1}\leq...\leq Z_n^{x_m}$. Let $\beta$ be any permutation of a set $\{1, 2,..., m-1\}$. Then:
$$\sum_{j=1}^{m-1}\frac{|B_nf(x_{j+1})-B_nf(x_j)|}{\lambda_{\beta(j)}}=\sum_{j=1}^{m-1}\frac{|\mathbb{E}f(Z_n^{x_{j+1}})-\mathbb{E}f(Z_n^{x_j})|}{\lambda_{\beta(j)}}\leq$$
$$\leq \mathbb{E}\left(\sum_{j=1}^{m-1}\frac{|f(Z_n^{x_{j+1}})-f(Z_n^{x_j})|}{\lambda_{\beta(j)}}\right)$$
Now, $Z_n^{x_1}(\omega)\leq...\leq Z_n^{x_m}(\omega)$ are numbers from interval $[0,1]$ for all $\omega\in\Omega$ and hence for every $\omega\in\Omega$:
$$\sum_{j=1}^{m-1}\frac{|f(Z_n^{x_{j+1}}(\omega))-f(Z_n^{x_j}(\omega))|}{\lambda_{\beta(j)}}\leq  V_\Lambda(f)$$
and therefore:
$$\mathbb{E}\left(\sum_{j=1}^{m-1}\frac{|f(Z_n^{x_{j+1}})-f(Z_n^{x_j})|}{\lambda_{\beta(j)}}\right)\leq \int_{\Omega}V_\Lambda(f)dP=V_\Lambda(f).$$
Finally:
$$\sum_{j=1}^{m-1}\frac{|B_nf(x_{j+1})-B_nf(x_j)|}{\lambda_{\beta(j)}} \leq V_\Lambda(f)$$
for arbitrary $0\leq x_1\leq...\leq x_m\leq 1$ and permutation $\beta$, what implies that $V_\Lambda(B_nf)\leq V_\Lambda(f)$.
\end{proof}

Now we are going to conclude from Theorem \ref{Bernstein_operators_diminish_lambda_variation} the $\Lambda$-variation diminishing property of the Kantorovich polynomials. Recall that if $f\in\Lambda BV$, then $f$ is a bounded Borel function, therefore it is integrable, hence the Kantorovich polynomials of this function may be considered. We start with an easy observation:

\begin{prop}
\label{prop2}
If $f\in L^1[0,1]$, $[a,b]\subseteq [0,1]$ then there exist $x_1, x_2\in (a,b)$ such that:
$$\frac{1}{b-a}\int_a^bf(t)dt\leq f(x_1),\qquad and\\ \qquad \frac{1}{b-a}\int_a^bf(t)dt\geq f(x_2).\\$$
\end{prop}

We say that a point $x\in (0,1)$ is a point of varying monotonicity of a function $f:[0,1]\rightarrow\mathbb{R}$ if there is no neighbourhood of $x$ on which $f$ is strictly monotone or constant. The points 0 and 1 are said to be points of varying monotonicity of $f$ if $f$ is non-constant on every neighbourhood of 0 or 1, respectively. The set of all points of varying monotonicity of a function $f$ will be denoted by $M_f$. According to \cite{Prus_Variation_Hausdorff_distance}, this concept was introduced in \cite{Bruckner_Differentiability_change_variables}.

\begin{thr}
\label{Kantorovich_operators_diminish_lambda_variation}
If $\Lambda$ is a $\Lambda$-sequence, $f\in\Lambda BV$, then $V_\Lambda(K_nf)\leq V_\Lambda(f)$.
\end{thr}
\begin{proof}
Let us define a function $g$ as a piecewise linear function such that:
$$g\left(\frac{k}{n}\right)=(n+1)\int_{\frac{k}{n+1}}^{\frac{k+1}{n+1}}f(t)dt,\quad k=0,1,...,n\\$$
It is clear that $K_nf=B_ng$ and therefore it sufficies to prove that $V_\Lambda(g)\leq V_\Lambda(f)$, since $V_\Lambda(K_nf)=V_\Lambda(B_ng)\leq V_\Lambda(g)$ from Theorem \ref{Bernstein_operators_diminish_lambda_variation}.

It is trivial if $g$ is constant. Take any partition $0=x_1<...<x_m=1$ of $[0,1]$ such that $|g(x_{i+1})-g(x_i)|>0$ for $i=1,...,m-1$. We may assume that $x_k\in M_g$ for $k=1,...,m$ (see Proposition 1.1 in \cite{Prus_Variation_Hausdorff_distance}) and therefore that $x_k=\frac{i_k}{n}$ for certain natural $i_k\leq n$ and every $k=1,...,m$, since only such points are points of varying monotonicity of $g$. Moreover, we may assume that for $k=1,...,m-2$ the sequence $g(\frac{i_k}{n}), g(\frac{i_{k+1}}{n}), g(\frac{i_{k+2}}{n})$ is not monotone. Indeed, if it is monotone then for an arbitrary permutation $\beta$ of the set $\{1,...,m-1\}$:
$$\frac{|g(\frac{i_{k+1}}{n})-g(\frac{i_{k}}{n})|}{\lambda_{\beta(k)}}+\frac{|g(\frac{i_{k+2}}{n})-g(\frac{i_{k+1}}{n})|}{\lambda_{\beta(k+1)}}\leq \frac{|g(\frac{i_{k+2}}{n})-g(\frac{i_{k}}{n})|}{\min\{\lambda_{\beta(k+1)}, \lambda_{\beta(k)}\}}.$$

Now, take $k=1,...,m$. If $g(\frac{i_k}{n})$ is greater than values of $g$ at neighbouring points of our partition, then using Proposition \ref{prop2} we define $t_k$ as a point from $(\frac{i_k}{n+1}, \frac{i_{k}+1}{n+1})$ such that:
$$f(t_k)\geq (n+1)\int_{\frac{i_k}{n+1}}^{\frac{i_{k}+1}{n+1}}f(t)dt=g\left(\frac{i_k}{n}\right).$$
Similarly, if $g(\frac{i_k}{n})$ is less than values of $g$ at neighbouring points of our partition, then we define $t_k$ as a point from $(\frac{i_k}{n+1}, \frac{i_{k}+1}{n+1})$ such that:
$$f(t_k)\leq (n+1)\int_{\frac{i_k}{n+1}}^{\frac{i_{k}+1}{n+1}}f(t)dt=g\left(\frac{i_k}{n}\right).$$
Obviously, $0\leq t_1\leq t_2\leq...\leq t_m\leq 1$ and:
$$\sum_{k=1}^{m-1} \frac{|g(\frac{i_{k+1}}{n})-g(\frac{i_{k}}{n})|}{\lambda_{\beta(k)}}\leq \sum_{k=1}^{m-1} \frac{|f(t_{k+1})-f(t_{k})|}{\lambda_{\beta(k)}}.$$
The initial partition and the permutation $\beta$ were arbitrary, hence $V_\Lambda(g)\leq V_\Lambda(f)$.
\end{proof}

\section{Other notions of $\Lambda$-variation}
Apart from the regular $\Lambda$- and $\varphi$-variation presented in the previous sections there are considered also other, related concepts of variation. Some of them will be described below.

If $\mathcal{I}$ is a finite sequence of intervals contained in $[0,1]$, then let us denote by $\|\mathcal{I}\|$ the length of the longest interval in this sequence. Further, define for a function $f:[0,1]\rightarrow\mathbb{R}$:
$$V_{\Lambda, \delta}(f)=\sup\sigma_\Lambda \left(f,\mathcal{I}\right)$$
where the supremum is taken over all finite sequences $\mathcal{I}$ of nonoverlapping closed intervals with $\|\mathcal{I}\|\leq\delta$, contained in $[0,1]$. The following quantity is called the Wiener $\Lambda$-variation of a function $f:[0,1]\rightarrow\mathbb{R}$:
$$W_\Lambda(f)=\lim_{\delta\to 0+} V_{\Lambda, \delta}(f).$$
Similarly, define:
$$V_{\varphi, \delta}(f)=\sup \sum_{i=1}^{n-1} \varphi\left(|f(x_{i+1})-f(x_{i})|\right)$$
where supremum is taken over all finite sequences $0\leq x_1<...<x_n\leq 1$ with $x_{i+1}-x_{i}\leq\delta$. The value:
$$V^*_\varphi(f)=\lim_{\delta\to 0+} V_{\varphi, \delta}(f).$$
is called the fine $\varphi$-variation of $f$.

In \cite{Adell_Cal_Bernstein_operators_diminish_variation} the authors prove under some assumptions on $\varphi$ and $f$ that:
$$V^*_\varphi(B_nf)\leq V^*_\varphi(f)$$
This observation becomes trivial in the case of the $\Lambda$-variation of a function defined on the interval $[0,1]$, due to the following proposition:

\begin{prop}
\label{polynomials_wiener_variation}
If $\Lambda$ is a proper $\Lambda$-sequence, $f:[0,1]\rightarrow\mathbb{R}$ is a Lipschitz function with Lipschitz constant $L$, then $W_\Lambda(f)=0$.
\end{prop}
\begin{proof} Take $\varepsilon>0$. Let $m$ be such that $\frac{L}{\lambda_m}<\frac{\varepsilon}{2}$ and $\delta>0$ such that $\frac{mL\delta}{\lambda_1}<\frac{\varepsilon}{2}$. Let $\mathcal{I}=\langle I_1,..., I_k\rangle$ be any finite sequence of intervals such that $\|\mathcal{I}\|\leq\delta$. If $k>m$ then:
$$\sum_{j=1}^k\frac{|f(I_j)|}{\lambda_j}=\sum_{j=1}^m\frac{|f(I_j)|}{\lambda_j}+\sum_{j=m+1}^k\frac{|f(I_j)|}{\lambda_j}\leq$$
$$\leq \frac{mL\delta}{\lambda_1}+\frac{L}{\lambda_m}\sum_{j=m+1}^k|I_j|\leq\frac{mL\delta}{\lambda_1}+\frac{L}{\lambda_m} <\varepsilon.$$
Similarly if $k\leq m$.
\end{proof}

One may ask if it is true that $V_{\delta, \Lambda}(B_nf)\leq V_{\delta, \Lambda}(f)$ for every $\delta>0$ and natural $n$. Suprisingly, we can give a counterexample to this claim. Let:

$$
f(x) = \left\{ \begin{array}{ll}
1.5x & \textrm{$0\leq x\leq \frac{1}{3}$}\\
0.5 & \textrm{$\frac{1}{3}<x\leq\frac{2}{3}$}\\
1.5x-0.5 & \textrm{$\frac{2}{3}<x\leq 1$}
\end{array} \right.
$$

The function $f$ is a piecewise linear function such that $f(0)=0, f(\frac{1}{3})=0.5, f(\frac{2}{3})=0.5, f(1)=1$. We shall prove that $V_{\delta, \Lambda}(B_nf)> V_{\delta, \Lambda}(f)$ for every natural $n$, $\Lambda$ such that $\lambda_1<\lambda_2$ and $1>\delta>\frac{2}{3}$. 

First, we are going to prove that the partition $\mathcal{I}=\langle [0,\delta], [\delta,1]\rangle$ has the property that $\sigma_\Lambda(f,\mathcal{I})=V_{\Lambda, \delta}(f)$. Let $\omega$ be the modulus of continuity of $f$, $\omega(\varepsilon)=\sup\{|f(t_1)-f(t_2)|: |t_1-t_2|\leq\varepsilon\}$ for $\varepsilon>0$. Observe that from the monotonicity and the continuity of $f$, there exists $x_0\leq 1-\delta<\frac{1}{3}$ such that $\omega(\delta)=f(x_0+\delta)-f(x_0)$. Obviously, $x_0+\delta\geq \frac{2}{3}$ and therefore:
$$\omega(\delta)=f(x_0+\delta) - f(x_0) = 1.5(x_0+\delta)-0.5 - 1.5 x_0 = 1.5\delta - 0.5 = |f(\delta)-f(0)|.$$

\begin{obs}
\label{observation_optimal_partition}
If $\mathcal{I'}=\langle I_1, I_2,..., I_m\rangle$ is an arbitrary sequence of nonoverlapping, closed intervals with $\|\mathcal{I'}\|\leq \delta$, then: $$\sigma_\Lambda(f,\mathcal{I'})\leq\sigma_\Lambda(f,\mathcal{I})=\frac{|f(\delta)-f(0)|}{\lambda_1}+\frac{|f(1)-f(\delta)|}{\lambda_2}.$$
\end{obs}
\begin{proof} Obviously $|f(I_1)| \leq \omega(\delta)=|f(\delta)-f(0)|$ and therefore:
$$\frac{|f(\delta)-f(0)|-|f(I_1)|}{\lambda_1}\geq \frac{|f(\delta)-f(0)|-|f(I_1)|}{\lambda_2}.$$
From the monotonicity of $f$ we get $|f(I_1)|+|f(I_2)|+...+|f(I_m)|\leq 1$ and $|f(\delta)-f(0)|+|f(1)-f(\delta)|=1$. Using both this facts:
$$\frac{|f(\delta)-f(0)|}{\lambda_1}+\frac{|f(1)-f(\delta)|}{\lambda_2}=$$
$$=\frac{|f(I_1)|}{\lambda_1}+\frac{|f(1)-f(\delta)|}{\lambda_2}+\frac{|f(\delta)-f(0)|-|f(I_1)|}{\lambda_1}\geq$$
$$\geq \frac{|f(I_1)|}{\lambda_1}+\frac{|f(1)-f(\delta)|}{\lambda_2}+\frac{|f(\delta)-f(0)|-|f(I_1)|}{\lambda_2}=$$
$$=\frac{|f(I_1)|}{\lambda_1}+\frac{1-|f(I_1)|}{\lambda_2}\geq$$
$$\geq\frac{|f(I_1)|}{\lambda_1}+\frac{|f(I_2)|+|f(I_3)|+...+|f(I_m)|}{\lambda_2}\geq \sigma_\Lambda(f,\mathcal{I'}).\qedhere$$
\end{proof}

\noindent Observation \ref{observation_optimal_partition} allows us to claim that:
\begin{equation}
\label{optimal_partition}
V_{\delta, \Lambda}(f)=\frac{|f(\delta)-f(0)|}{\lambda_1}+\frac{|f(1)-f(\delta)|}{\lambda_2}
\end{equation}

\begin{obs}
\label{inequality_values_delta}
$B_nf(\delta)>f(\delta).$
\end{obs}
\begin{proof} Let us notice that the Bernstein operators are positive, i. e. nonnegativity of a function $h$ implies nonnegativity of $B_nh$ and therefore if we put $g(x)=1.5x-0.5$, then $f-g\geq 0$ and $B_n(f-g)\geq 0$. From linearity of the Bernstein operator $B_nf\geq B_ng=g$. The last equality follows from the fact that $B_nh=h$ for an arbitrary linear function $h$ (see the beginning of the chapter 3.2 in \cite{Farouki_Bernstein_polynomial_basis}). In particular we have $B_nf(\delta)\geq f(\delta)=g(\delta)$, since $\delta>\frac{2}{3}$.

Now it sufficies to show that the assumption $B_nf(\delta)=f(\delta)$ leads to a contradiction. Note that this equality implies that $B_nf(\delta)=g(\delta)=B_ng(\delta)$. Using the monotonicity of $f-g$ and the preservation of monotonicity by the Bernstein polynomials (Proposition \ref{preservation_monotonicity}) we know that $B_n(f-g)(\delta)\geq B_n(f-g)(x)\geq 0$ for $x\geq\delta$, but then $B_n(f-g)(x)=0$ for $x\geq \delta$, since $B_n(f-g)(\delta)=0$, and thereby $B_n(f-g)(x)=0$ for all $x\in [0,1]$. But observe that $f(0)>g(0)$, hence $B_nf(0)>B_ng(0)$ directly from the definition of the Bernstein polynomials, what is a desired contradiction.
\end{proof}

\noindent From Observation \ref{inequality_values_delta}:
\begin{equation}
\label{final_inequality}
\left|B_nf(\delta)-B_nf(0)\right|-\left|f(\delta)-f(0)\right|=\big(B_nf(\delta)-B_nf(0)\big)-\big(f(\delta)-f(0)\big)>0
\end{equation}
Using the monotonicity of $f$ and the preservation of monotonicity by $B_n$:
\begin{equation}
\label{sum_up_to}
|f(1)-f(\delta)|+|f(\delta)-f(0)|=1 \qquad |B_nf(1)-B_nf(\delta)|+|B_nf(\delta)-B_nf(0)|=1\\
\end{equation}

\noindent Now:
$$\left(\frac{|B_nf(\delta)-B_nf(0)|}{\lambda_1}+\frac{|B_nf(1)-B_nf(\delta)|}{\lambda_2}\right)-\left(\frac{|f(\delta)-f(0)|}{\lambda_1}+\frac{|f(1)-f(\delta)|}{\lambda_2}\right)=$$
$$=\frac{|B_nf(\delta)-B_nf(0)|-|f(\delta)-f(0)|}{\lambda_1}-\frac{|f(1)-f(\delta)|-|B_nf(1)-B_nf(\delta)|}{\lambda_2}>$$
$$>\frac{|B_nf(\delta)-B_nf(0)|-|f(\delta)-f(0)|}{\lambda_2}-\frac{|f(1)-f(\delta)|-|B_nf(1)-B_nf(\delta)|}{\lambda_2}=0.$$

\noindent In the inequality above we needed the assumption that $\lambda_1<\lambda_2$ and (\ref{final_inequality}). The last equality follows from (\ref{sum_up_to}). Finally $V_{\Lambda, \delta}(B_nf)>V_{\Lambda, \delta}(f)$ from (\ref{optimal_partition}).

\begin{remark}
In \cite{Bugajewska_Lower_variation} the authors defined so-called lower $\Lambda$-variation. For a real valued function $f$ on $[0,1]$ we define the lower $\Lambda$-variation of $f$ as:
$$\underline{\textrm{var}}_\Lambda(f)=\inf\{V_\Lambda(g): f=g\ \textrm{a.e.}\}$$
Let:
$$\underline{\Lambda BV}=\{f\in L^1: \underline{\textrm{var}}_\Lambda(f)<\infty\}.$$
This space is a Banach space with the norm $\|f\|_{\underline{\Lambda}}=\|f\|_1+\underline{\textrm{var}}_\Lambda(f)$. A function $f\in\Lambda BV$ is called a good representative of its equivalence class in the space $L^1$ if $V_\Lambda(f)=\underline{\textrm{var}}_\Lambda(f)$.

Note that the Kantorovich operator is well-defined on this space. Theorem \ref{Kantorovich_operators_diminish_lambda_variation} tells us that:
$$V_\Lambda(K_nf)\leq \underline{\textrm{var}}_\Lambda(f)$$
for $f\in\underline{\Lambda BV}$. Moreover, Theorem 8 in \cite{Bugajewska_Lower_variation} states that continuous functions are  good representatives and therefore:
$$\underline{\textrm{var}}_\Lambda(K_nf)\leq\underline{\textrm{var}}_\Lambda(f)$$
i.e. the Kantorovich operators diminish the lower $\Lambda$-variation.
\end{remark}

\section{Applications in $\Lambda BV$ spaces}
We begin this section with several propositions generalizing Lorentz's results described in \cite{Lorentz_Bernstein_polynomials}, Theorems 1.7.1 and 1.7.2. Let us define for nonempty $K\subseteq [0,1]$ and $f:[0,1]\rightarrow\mathbb{R}$:
$$V_\Lambda(f, K)=\sup\sigma_\Lambda\left(f,\mathcal{I}\right)$$
where supremum is taken over all finite sequences of nonoverlapping closed intervals whose endpoints are in $K$. The following statement is a straightforward generalization of the analogous result in the case of the regular variation. The elementary proof will be omitted.

\begin{prop}
\label{variation_pointwise_limit}
If $\Lambda$ is a $\Lambda$-sequence, $K\subseteq[0,1]$ is nonempty, $f_n(x)\to f(x)$ for every $x\in K$, then:
$$V_\Lambda(f, K)\leq\liminf_{n\to\infty} V_\Lambda (f_n, K).$$
In particular, if $f$ is not of bounded $\Lambda$-variation, then $\lim_{n\to\infty} V_\Lambda (f_n)=\infty$.
\end{prop}
Obviously, the Bernstein polynomials of a continuous function $f$ tend to $f$ pointwise and therefore we have the following proposition. Recall that $\|f\|_\Lambda=V_\Lambda(f)+|f(0)|$.
\begin{prop}
\label{continuity_lambda_norm_continuous}
If $\Lambda$ is a $\Lambda$-sequence, $f\in C\Lambda BV$, then:
$$\lim_{n\to\infty} \|B_nf\|_\Lambda=\|f\|_\Lambda.$$
\end{prop}
\begin{proof} Theorem \ref{Bernstein_operators_diminish_lambda_variation} implies that $\limsup_{n\to\infty}V_\Lambda(B_nf)\leq V_\Lambda(f)$. Proposition \ref{variation_pointwise_limit}, however, yields that $\liminf_{n\to\infty} V_\Lambda(B_nf)\geq V_\Lambda(f)$.  The conclusion follows, since $B_nf(0)=f(0)$ for every natural $n$.
\end{proof}

\begin{prop}
\label{charaterization_continuous_lambda_bv}
If $\Lambda$ is a $\Lambda$-sequence, then:
$$C\Lambda BV=\{f\in C[0,1]: \sup_{n\in\mathbb{N}}\|B_nf\|_\Lambda<\infty\}.$$
\end{prop}
\begin{proof} If $f\in\Lambda BV$, then $\sup_n\|B_nf\|_\Lambda\leq \|f\|_\Lambda$. If $f\not\in\Lambda BV$, then $\lim_{n\to\infty} V_\Lambda (B_nf)=\infty$ from Proposition \ref{variation_pointwise_limit}.
\end{proof}

In Propositions \ref{continuity_lambda_norm_continuous} and \ref{charaterization_continuous_lambda_bv} we had to restrict to continuous functions, since discontinuous functions  in general do not have the property that $B_nf\to f$ pointwise. In the case of discontinuous functions it is necessary to add some further assumptions on $f$.

For arbitrary $f:[0,1]\rightarrow\mathbb{R}$, let $K_f$ be the set of points of continuity of $f$, and 0 and 1. Further, denote by $\mathfrak{A}$ the set of all bounded functions with one-sided limits at every point of $(0,1)$ and such that $\min\{f(x-), f(x+)\}\leq f(x)\leq\max\{f(x-), f(x+)\}$. If $f\in\Lambda BV$ for some $\Lambda$-sequence $\Lambda$, then $f$ admits the first one of these properties but not necessarily the second one (Theorem 4 in \cite{Perlman_Functions_of_generalized_variation}). Observe that from the first property we may conclude that the set of points of discontinuity of a function from $\mathfrak{A}$ is at most countable. Indeed, if $f\in\mathfrak{A}$, then this set is the union of the sets $E_n=\{x\in[0,1]:|f(x+)-f(x-)|>\frac{1}{n}\}$ for $n=1,2,...$ and each $E_n$ is finite from the compactness of $[0,1]$ and the definintion of the family $\mathfrak{A}$.

We are ready to formulate the next lemma.

\begin{lem}
\label{variation_points_continuity}
If $\Lambda$ is a $\Lambda$-sequence, $f\in\mathfrak{A}$, then $V_\Lambda(f, K_f)=V_\Lambda(f)$.
\end{lem}
\begin{proof} It is trivial when $f$ is constant. Obviously, $V_\Lambda(f, K_f)\leq V_\Lambda(f)$. Let $0=x_1<...<x_m=1$ be an arbitrary partition of $[0,1]$ such that $|f(x_{k+1})-f(x_k)|>0$ for $k=1,...,m-1$. Using the same argument as in the proof of Theorem \ref{Kantorovich_operators_diminish_lambda_variation} we may assume that  for $k=1,...,m-2$ the sequence $f(x_k), f(x_{k+1}), f(x_{k+2})$ is not monotone. Put:
$$\delta=\frac{1}{2}\min_{i=1,...,m-1}\left|x_{i+1}-x_{i}\right|.$$
Take $\varepsilon>0$ such that:
\begin{equation}
\label{epsilon}
\varepsilon<\min_{i=1,...,m-1} \frac{|f(x_{i+1})-f(x_i)|}{\lambda_1}.
\end{equation}
We construct a new partition $t_1<t_2<...<t_m$ in the following way: let $t_1=x_1$, $t_m=x_m$. If $x_k$ is a point of continuity of $f$, then put $t_k=x_k$, $k=2,...,m-1$. If $x_k$ for $k=2,...,m-1$ is a point of discontinuity, then let $t_k$ be any point of continuity such that $|x_k-t_k|<\delta$ and:
\begin{enumerate}[(i)]
\item $f(t_k)\geq f(x_k)-\frac{\lambda_1\varepsilon}{2(m-1)}$ if $f(x_k)$ is greater than values of $f$ at neighbouring points of the partition $x_1<...<x_m$
\item $f(t_k)\leq f(x_k)+\frac{\lambda_1\varepsilon}{2(m-1)}$ otherwise
\end{enumerate}
It is possible, since $\limsup_{t\to x_k}f(t)\geq f(x_k), \liminf_{t\to x_k}f(t)\leq f(x_k)$ for $k=2,...,m-1$ and the set of points of discontinuity of $f$ is at most countable.

For every $k$ we have $|x_k-t_k|<\delta$, hence $t_1<t_2<...<t_m$ is a partition of $[0,1]$. Observe that:
$$\frac{\lambda_1\varepsilon}{m-1}\sum_{i=1}^{m-1}\frac{1}{\lambda_{\beta(j)}}\leq \frac{\lambda_1\varepsilon}{m-1}\cdot\frac{m-1}{\lambda_1}=\varepsilon.$$
 Further, (\ref{epsilon}) implies that for $k=1,...,m-2$ the sequence $f(t_k), f(t_{k+1}), f(t_{k+2})$ is not monotone and therefore for an arbitrary permutation $\beta$ of the set $\{1,...,m-1\}$ we have:
$$\sum_{j=1}^{m-1}\frac{|f(t_{j+1})-f(t_j)|}{\lambda_{\beta(j)}}\geq \sum_{j=1}^{m-1}\frac{|f(x_{j+1})-f(x_j)|}{\lambda_{\beta(j)}}-\frac{\lambda_1\varepsilon}{m-1}\sum_{i=1}^{m-1}\frac{1}{\lambda_{\beta(j)}}\geq$$
$$\geq\sum_{j=1}^{m-1}\frac{|f(x_{j+1})-f(x_j)|}{\lambda_{\beta(j)}}-\varepsilon.$$
Since the inequality remains valid for any partition $x_1<...<x_m$ and any $\varepsilon > 0$ we can observe that $V_\Lambda(f, K_f)\geq V_\Lambda(f)$.
\end{proof}
\noindent Proofs of the next two propositions follow from Lemma \ref{variation_points_continuity}, Proposition \ref{variation_pointwise_limit}, Theorem \ref{Bernstein_operators_diminish_lambda_variation} and the fact that $B_nf(x)\to f(x)$ at every point of continuity $x$ of the function $f$ (Theorem 1.1.1. in \cite{Lorentz_Bernstein_polynomials}).
\begin{prop}
If $\Lambda$ is a $\Lambda$-sequence, $f\in\mathfrak{A}$, then:
$$\lim_{n\to\infty} \|B_nf\|_\Lambda=\|f\|_\Lambda.$$
\end{prop}
\begin{prop}
If $\Lambda$ is a $\Lambda$-sequence, $f\in\mathfrak{A}$, then $f$ is of bounded $\Lambda$-variation if and only if $\sup_n\|B_nf\|_\Lambda<\infty$.
\end{prop}

Let us now concentrate on the problem of a convergence of sequences $(B_nf)_n$ and $(K_nf)_n$ to $f$ in the $\|\cdot\|_\Lambda$ norm and thereby on the main result of this section. First, let us give the proof of the following proposition.

\begin{prop}
If $\Lambda$ is a proper $\Lambda$-sequence, then $\Lambda BV_c$ and $C\Lambda BV_c$ are closed subspaces of $\Lambda BV$ in the $\|\cdot\|_\Lambda$ norm.
\end{prop}
\begin{proof}
For every $f\in\Lambda BV$ we have $V_{\Lambda_{(1)}}(f)\geq V_{\Lambda_{(2)}}(f)\geq...$. We deduce from this observation that if $f\not\in\Lambda BV_c$, then there exists $\varepsilon>0$ such that $V_{\Lambda_{(m)}}(f)\geq \varepsilon$ for every natural $m$.
Take an arbitrary $g\in\Lambda BV$ such that $\|f-g\|_\Lambda<\frac{\varepsilon}{2}$. Obviously, $V_{\Lambda_{(m)}}(f-g)<\frac{\varepsilon}{2}$  for every natural $m$ and therefore $V_{\Lambda_{(m)}}(g)\geq\frac{\varepsilon}{2}$ for every natural $m$ from the reverse triangle inequality (recall that the $\Lambda$-variation is a seminorm). However, this implies that $g\not\in\Lambda BV_c$, what proves that $\Lambda BV_c$ is a closed subspace of $\Lambda BV$. The space $C\Lambda BV_c$ is a closed subspace as an intersection of two closed subspaces.
\end{proof}

As it was mentioned in the introduction, the space $C\Lambda BV_c$ appeared in the theory of the $\Lambda$-variation mainly due to good properties of Fourier series of functions from this space. This led Waterman to ask about a characterization of the space $\Lambda BV_c$ (see \cite{Waterman_On_Lambda_bounded_variation}). He also expressed a suspicion (see \cite{Waterman_Bounded_variation_Fourier_series}) that not every function of bounded $\Lambda$-variation is continuous in $\Lambda$-variation. An example of such function was given initially in \cite{Fleissner_Note_Lambda_bounded_variation} but in fact the most beautiful confirmation of Waterman's conjecture was provided by F. Prus-Wi\'sniowski in \cite{Prus_Separability_continuous_functions_contunuous_Lambda_variation}, where he defined the Shao-Sablin index $S_\Lambda$ of a $\Lambda$-sequence $\Lambda$ as:
$$S_\Lambda=\limsup_{n\to\infty}\frac{\sum_{i=1}^{2n}\frac{1}{\lambda_i}}{\sum_{i=1}^{n}\frac{1}{\lambda_i}}$$
and proved the theorem (Theorem 3.1 in \cite{Prus_Separability_continuous_functions_contunuous_Lambda_variation}):
\begin{thr}[Prus-Wi\'sniowski] 
If $\Lambda$ is a proper $\Lambda$-sequence, then the following statements are equivalent:
\begin{enumerate}[(i)]
\item The space $C\Lambda BV$ is separable
\item $C\Lambda BV_c=C\Lambda BV$
\item $\Lambda BV_c=\Lambda BV$
\item $S_\Lambda<2.$
\end{enumerate}
\end{thr}

Several other characterizations of $\Lambda BV_c$ and $C\Lambda BV_c$ being an answer or partial answer to Waterman's question from \cite{Waterman_On_Lambda_bounded_variation} have been published over last three decades, for instance:

\begin{thr}[\cite{Wang_Properties_functions_Lambda_bounded_variation}]
If $\Lambda$ is a proper $\Lambda$-sequence, then $f\in\Lambda BV_c$ if and only if there exists $\Lambda$-sequence $\Gamma$ such that $\frac{\gamma_n}{\lambda_n}\to 0$ and $f\in\Gamma BV$.
\end{thr}

\begin{thr}[\cite{Prus_Absolute_continuity}, Theorem 2]
\label{polynomials_continuous_variation}
If $\Lambda$ is a proper $\Lambda$-sequence, then $f\in C\Lambda BV_c$ if and only if $W_\Lambda(f)=0$.
\end{thr}
\begin{thr}[\cite{Prus_Absolute_continuity}, Theorem 3]
If $\Lambda$ is a proper $\Lambda$-sequence, then $\Lambda BV_c$ is the closure of the set of all step functions of bounded $\Lambda$-variation in the $\|\cdot\|_\Lambda$-norm.
\end{thr}
For more theorems in this spirit involving also other notions of the $\Lambda$-variation see \cite{Prus_Absolute_continuity}. Now we give the proof of the main result of this section with three interesting corollaries.

\begin{thr}
\label{convergence_characterization}
If $\Lambda$ is a proper $\Lambda$-sequence, then $||B_nf-f||_\Lambda\to 0$ if and only if $f\in C\Lambda BV_c$.
\end{thr}
\begin{proof}
$C\Lambda BV_c$ is a closed subspace of $\Lambda BV$ containing the space of all polynomials, what may be concluded from Theorem \ref{polynomials_continuous_variation} and Proposition \ref{polynomials_wiener_variation}. It implies that if $||B_nf-f||_\Lambda\to 0$ then $f\in C\Lambda BV_c$.

Assume that $f\in C\Lambda BV_c$. Take $\varepsilon>0$. There exists such $m$ that $V_{\Lambda_{(m)}}(f)\leq \frac{\varepsilon}{3}$. The $\Lambda$-variation diminishing property implies that $V_{\Lambda_{(m)}}(B_nf)\leq \frac{\varepsilon}{3}$ for every natural $n$. Moreover, there exists $N\in\mathbb{N}$ such that:
$$||B_nf-f||_\infty\leq\frac{\varepsilon\lambda_1}{6m}$$
for $n\geq N$.

Let $\langle I_1,..., I_k\rangle$ be any sequence of nonoverlapping intervals contained in $[0,1]$. If $k>m$, then:
$$\sum_{j=1}^k\frac{|(B_nf-f)(I_j)|}{\lambda_j}=\sum_{j=1}^m\frac{|(B_nf-f)(I_j)|}{\lambda_j}+\sum_{j=m+1}^k\frac{|(B_nf-f)(I_j)|}{\lambda_j}\leq $$
$$\leq\frac{1}{\lambda_1}\sum_{j=1}^m|(B_nf-f)(I_j)|+\sum_{j=m+1}^k\frac{|B_nf(I_j)|}{\lambda_j}+\sum_{j=m+1}^k\frac{|f(I_j)|}{\lambda_j}\leq$$
$$\leq\frac{1}{\lambda_1}\cdot 2m\cdot \frac{\varepsilon\lambda_1}{6m}+ V_{\Lambda_{(m)}}(B_nf)+V_{\Lambda_{(m)}}(f)\leq \varepsilon$$
for $n\geq N$. Similarly if $k\leq m$. Hence, using also fact that $\left|(B_nf-f)(0)\right|=0$:
$$||B_nf-f||_\Lambda\leq\varepsilon$$
for $n\geq N$. Since $\varepsilon$ was arbitrary, our theorem is proved.
\end{proof}

\begin{remark}
In the proof of Theorem \ref{convergence_characterization} only two properties of the Bernstein polynomials were significant: that $\|B_nf-f\|_\infty\to 0$ for any continuous function $f$ and that the Bernstein polynomials diminish the $\Lambda$-variation. The Kantorovich polynomials also have these properties. Indeed, the first one follows from chapter 10, paragraph 6 in \cite{Lorentz_Constructive_approximation} and the second one from Theorem \ref{Kantorovich_operators_diminish_lambda_variation}. Therefore, we may prove in a very similar way that $||K_nf-f||_\Lambda\to 0$ if and only if $f\in C\Lambda BV_c$.
\end{remark}

\begin{cor}
If $\Lambda$ is a proper $\Lambda$-sequence, then $C\Lambda BV_c$ is the clousure of the space of polynomials in the $\|\cdot\|_\Lambda$ norm.
\end{cor}
\noindent This is a partial answer to Waterman's question from \cite{Waterman_On_Lambda_bounded_variation}.

Before the next corollary let us recall that for every function $f\in C^1[0,1]$, in particular for every polynomial, we have:
\begin{equation}
\label{expression_variation_derivative}
V(f)=\int_0^1\left|f'(x)\right|dx
\end{equation}
\noindent The first proof of the next corollary was given by Prus-Wi\'sniowski in \cite{Prus_Variation_Hausdorff_distance} and was rather technical and long. Observe that in our proof of this statement we use only Theorem \ref{Bernstein_operators_diminish_lambda_variation} and Theorem \ref{convergence_characterization}.
\begin{cor}
If $\Lambda$ is a proper $\Lambda$-sequence, then $C\Lambda BV_c$ is separable.
\end{cor}
\begin{proof} Obviously, now it sufficies to prove that the space of polynomials is separable in the $\|\cdot\|_\Lambda$ norm. We may easily show using (\ref{expression_variation_derivative}) that for every polynomial $f$ and $\varepsilon>0$ there exists a polynomial $g$ with rational coefficients such that $V(f-g)<\varepsilon$. For $f, g\in BV$ we have that $V_\Lambda(f-g)\leq \lambda_1^{-1} V(f-g)$ and the conclusion follows.
\end{proof}

The Bernstein operators are linear operators from $\Lambda BV$ to $\Lambda BV$. Moreover, Theorem \ref{Bernstein_operators_diminish_lambda_variation} tells us that these operators are also continuous and $\|B_n\|\leq 1$, where $\|\cdot\|$ denotes the standard norm in the space of linear operators from $\Lambda BV$ to $\Lambda BV$. For every constant function $f$ we have $B_nf=f$ (the beginning of the chapter 3.2 in \cite{Farouki_Bernstein_polynomial_basis}), what guarantees that $B_n$ is not a contraction and therefore $\|B_n\|=1$. The Bernstein operators are finite dimensional, continuous operators and hence are compact.

Now we use this fact to prove the characterization of compactness. At the end of chapter 2.1 of \cite{Lorentz_Bernstein_polynomials} the author gives the criterion of compactness in $L^p$ spaces based on the very similar idea with Kantorovich polynomials. Another characterization was provided by Prus-Wi\'sniowski (cf. \cite{Prus_Variation_Hausdorff_distance}). Some other results about compactess in $\Lambda BV$ spaces were recently given also by Bugajewski et al (see \cite{Bugajewski_On_integral_operators_nonlinear_operators}).

\begin{cor}
If $\Lambda$ is a proper $\Lambda$-sequence, $K\subseteq C\Lambda BV_c$ is bounded and closed, then $K$ is compact if and only if $\|B_nf-f\|_\Lambda\to 0$ uniformly for all $f\in K$.
\end{cor}
\begin{proof}
Assume that $\|B_nf-f\|_\Lambda\to 0$ uniformly. Since $K$ is closed, it sufficies to show that it is totally bounded. Set $\varepsilon>0$. There exists $n$ such that $\|B_nf-f\|_\Lambda<\frac{\varepsilon}{2}$ for $f\in K$. Operator $B_n$ is compact, $K$ is bounded and hence $B_nK$ is totally bounded. There exists $f_1, f_2,..., f_m\in K$ such that $B_nK\subseteq \bigcup_{i=1}^m B\left(B_nf_i, \frac{\varepsilon}{2}\right)$. If $f\in K$, then there exists $k$ such that $\|B_nf-B_nf_k\|_\Lambda<\frac{\varepsilon}{2}$, hence $\|B_nf_k-f\|_\Lambda<\varepsilon$. Since $\varepsilon$ was arbitrary, we have that $K$ is totally bounded.

Now, assume that $K$ is compact. There exist $f_1,..., f_k\in K$ such that $K\subseteq \bigcup_{i=1}^k B\left(f_i, \frac{\varepsilon}{3}\right)$. Let $N$ be such that $\|B_nf_i-f_i\|_{\Lambda}<\frac{\varepsilon}{3}$ for $i=1,2,...,.k$, $n\geq N$. Let $f\in K$. There exists $m$ such that $f\in B\left(f_m, \frac{\varepsilon}{3}\right)$. Now for $n\geq N$:
$$\|B_nf-f\|_\Lambda\leq \|B_n\|\|f-f_m\|_\Lambda+\|B_nf_m-f_m\|_\Lambda+\|f_m-f\|_\Lambda<\varepsilon.\hfill\qedhere$$
\end{proof}
\begin{acknowledgment}
I am grateful to Jacek Gulgowski for suggesting the topic of this paper and much valuable advice. I also would like to thank the referee for his constructive comments.
\end{acknowledgment}


\end{document}